\documentclass[a4paper]{amsart}
\usepackage[english]{babel}

\usepackage{amssymb, amsmath}
\usepackage{mathrsfs}
\usepackage{amscd}
\usepackage[active]{srcltx}
\usepackage{verbatim}
\usepackage{graphicx}
\usepackage[colorlinks,linkcolor={blue},citecolor={blue},urlcolor={red},]{hyperref}

\usepackage[fleqn,tbtags]{mathtools}
\mathtoolsset{showonlyrefs,showmanualtags}

\newcommand\dd{\mathrm{d}}
\newcommand\ud{\,\mathrm{d}}

\newcommand\Tr{\mathrm{Tr}}
\newcommand\EE{\mathbb{E}}
\newcommand\E{\mathbb{E}}
\newcommand\D{\mathsf{D}}
\newcommand\Ran{\mathsf{R}}
\newcommand\Ker{\mathsf{N}}
\newcommand{\RR}{\mathbb{R}}

\newcommand{\CC}{\mathbb{C}}
\newcommand{\Ll}{\mathscr{L}}
\newcommand{\inp}[2]{\langle #1,#2 \rangle}

\newcommand{\s}{^*}
\newcommand{\Ric}{\mathrm{Ric}}
\newcommand{\Nabla}{\nabla}
\DeclareMathOperator{\re}{Re}

\renewcommand{\Re}{{\rm Re}}

\theoremstyle{plain}
\newtheorem{theorem}{Theorem}[section]
\theoremstyle{remark}
\newtheorem{remark}[theorem]{Remark}
\newtheorem{example}[theorem]{Example}
\theoremstyle{plain}

\newtheorem{lemma}[theorem]{Lemma}
\newtheorem{proposition}[theorem]{Proposition}
\newtheorem{definition}[theorem]{Definition}
\newtheorem{hypothesis}[theorem]{Hypothesis}

\numberwithin{equation}{section}

\allowdisplaybreaks

\begin{document}

\title[The Hodge--Dirac operator associated with the Witten Laplacian]{$L^p$-Analysis of the 
Hodge--Dirac operator associated with Witten Laplacians on complete Riemannian 
manifolds}
\author{Jan van Neerven}
\author{Rik Versendaal}
\address{Delft Institute of Applied Mathematics//Delft University of 
Technology\\P.O. Box 5031, 2600 GA Delft\\The Netherlands}
\email{J.M.A.M.vanNeerven/R.Versendaal@TUDelft.nl}
\date\today

\begin{abstract}
We prove $R$-bisectoriality and boundedness of the $H^\infty$-functional calculus in $L^p$ for all $1<p<\infty$ for the
Hodge--Dirac operator associated with Witten Laplacians on complete Riemannian 
manifolds with non-negative Bakry--Emery Ricci curvature on $k$-forms.
\end{abstract}

\thanks{The first-named author acknowledges financial support from the ARC 
Discovery Grant DP 160100941. The
second-named author is supported by the Peter Paul Peterich Foundation via TU Delft University Fund.}

\maketitle
\section{Introduction}

The Witten Laplacian 
was introduced by Witten \cite{Witten} as a deformation of the Hodge Laplacian 
on a complete Riemannian
manifold $M$ and has been subsequently studied by many authors; see
\cite{Bakry, Bismut, CD-Bellman, ElwRos, FuLiLi, HelNier, HelSjo, Li05, Li08, 
Li10, Yoshida} and the references cited therein. 
The {\em Witten Laplacian} associated with a smooth strictly 
positive function $\rho:M\to \RR$ 
is the operator 
$$
L_\rho : f \mapsto   \Delta f - \nabla \log\rho\cdot \nabla f, \quad f\in C_{\rm c}^\infty(M),
$$
where $\Delta = \nabla^*\nabla$ is the (negative) Laplace-Beltrami operator and 
$\nabla$ is the gradient. Identifying functions with  
0-forms, we have
\begin{align}\label{eq:Lrho}
L_\rho f= (\dd_\rho\ud_\rho^* + \ud_\rho^*\ud_\rho)f, \quad f\in C_{\rm c}(M),
\end{align}
where $\ud_\rho$ is the $L^2$-realisation of the exterior derivative $\ud$ with 
respect to the measure 
$m(\dd x) = \rho(x) \ud x$ on $M$, and $\ud_\rho\s$ is the adjoint operator. The 
representation \eqref{eq:Lrho} can be used to define the Witten Laplacian for $k$-forms for $k 
\neq 0$.  
In the special case $M = \RR^n$ and $\rho(x) = \exp(-\frac12|x|^2)$, $L_\rho$ corresponds to 
the Ornstein-Uhlenbeck operator. 

Let $m({\rm d}x) = \rho(x)\ud x$ is the 
weighted volume measure on $M$. Generalising the celebrated Meyer inequalities
for the Ornstein-Uhlenbeck operator,
Bakry \cite{Bakry} proved boundedness of the Riesz transform $\nabla L_\rho^{-1/2}$ 
on $L^p(M,m)$ for all $1<p<\infty$ under a curvature condition on $M$.
An extension of this result to the 
corresponding $L^p$-spaces of $k$-forms is contained in the same paper. 
These results have been subsequently extended into various directions. As a sample 
of the extensive literature on this topic we 
mention \cite{CD-Bellman, Li05, Li08, Li10,  Yoshida} (for the Witten 
Laplacian); see also 
\cite{AC, ACDH, BO, Coulhon-Duong, LJ, LHQ, Lo, Magniez, Sikora, Strichartz} 
(for the Laplace-Beltrami operator),
\cite{CMO, HMM, MM} (for the Hodge-de Rham Laplacian), and \cite{BG} (for 
sub-elliptic operators).

The aim of the present paper is to develop Bakry's result along a different 
line by analysing 
the Hodge--Dirac operator 
$$D_\rho = \ud_\rho + \ud_\rho\s$$ 
from the point of view of its functional calculus properties. Our main result 
can be stated as follows 
(the relevant definitions are given in the main body of the paper). 

\begin{theorem}\label{thm:1}
If $M$ has non-negative Bakry--Emery Ricci curvature on $k$-forms for all $1\le k\le n$, then the Hodge--Dirac operator $D_\rho$ is $R$-bisectorial and admits a bounded $H^\infty$-calculus 
in $L^p(\Lambda TM,m)$ for all $1<p<\infty$.
\end{theorem}

By standard arguments (cf. \cite{AKM}), the boundedness of the $H^\infty$-calculus
of $D_\rho$ implies (by considering the operator sgn$(D_\rho)$, which is then 
well defined through the functional calculus) the boundedness of the Riesz transform
$D_\rho L_\rho^{-1/2} = {\rm sgn}(D_\rho)$.  
As such our results may be thought of as a strengthening of those in \cite{Bakry}.

In the unweighted case $\rho\equiv 1$, the second assertion of Theorem \ref{thm:1} is essentially known, although we are not 
aware of a place where it is formulated explicitly or in some equivalent form. It can be pieced together from known results as follows. Firstly, \cite[Theorem 5.12]{AMR}
asserts that the unweighted Hodge--Dirac operator $D$ has a bounded $H^\infty$-calculus on the Hardy space
$H^p(\Lambda TM)$, even for $1\le p\le \infty$, provided the volume measure has the so-called doubling property. 
By the Bishop comparison theorem (see \cite{BC}), this property is always satisfied if $M$ has non-negative Ricci curvature. 
Secondly, for $1<p<\infty$, this Hardy space is subsequently identified in  \cite[Theorem 8.5]{AMR} to be the closure in $L^p(\Lambda TM)$
of the range of $D$, 
provided the heat kernel associated with $L$ satisfies Gaussian bounds on $k$-forms for all $0\le k\le n$.
When $M$ has non-negative Ricci curvature, such bounds were proved
in \cite{LY} for $0$-forms, i.e., for functions on $M$. The bounds for $k$-forms then follow, under the curvature assumptions in the present paper, via pointwise domination
of the heat kernel on $k$-forms by the heat kernel for $0$-forms (cf. \eqref{eq:ptw} below).
Modulo the kernel-range decomposition decomposition $L^p(\Lambda TM,m) = \Ker(D) \oplus \overline{\Ran(D)}$ (which follows from $R$-bisectorialy proved in the present paper, but could also be established on the basis of other known results), this gives 
the boundedness of the $H^\infty$-calculus in $L^p(\Lambda TM,m)$ in the unweighted case. 

In the weighted case, this approach cannot be pursued due to the absence of doubling and 
Gaussian bounds. Instead, our approach exploits the fact, proved in \cite{Yoshida}, that the non-negativity of the Bakry--Emery Ricci curvature 
implies, among other things, square function estimates on $k$-forms.

The analogue of Theorem \ref{thm:1} for the Hodge--Dirac operator associated with the 
Orn\-stein-Uhlenbeck operator has been 
established, in a more general formulation, in \cite{MN}. The related problem of the $L^p$-boundedness of 
the $H^\infty$-calculus 
of Hodge--Dirac operators associated with the Kato square root problem was 
initiated by the influential 
paper \cite{AKM} and has been studied by many authors \cite{AS, FMP, HM, HMP1, 
HMP2, MM}.\smallskip

The organisation of the paper is as follows. After a brief introduction to 
$R$-(bi)sectorial 
operators and  $H^\infty$-calculi in Section 
\ref{sec:sectorial}, we introduce the Witten Laplacian $L_\rho$ in Section 
\ref{sec:WittenLaplacian}
and recall some of its properties. Among others we prove that it is 
$R$-sectorial of angle less than $\frac12\pi$ and admits a bounded 
$H^\infty$-calculus in $L^p$ for $1<p<\infty$. In
Section \ref{sec:HodgeDirac} this result,
together with the identity $D_\rho^2 = L_\rho$, is used to prove the 
corresponding assertions for the Hodge--Dirac 
operator $D_\rho$.
  
On some occasions we will use the notation $a\lesssim b$ to signify that there exists
a constant $C$ such that $a \le Cb$. To emphasise the dependence of $C$
on parameters $p_1$, $p_2$, \dots, we shall write $a \lesssim_{p_1,p_2,\dots} b$. 
Finally we write $\eqsim$ (respectively, $\eqsim_{p_1,p_2,\dots}$) if
both $a\lesssim b$ and $b\lesssim a$ (respectively, 
$a \lesssim_{p_1,p_2,\dots} b$ and $b \lesssim_{p_1,p_2,\dots} a$) hold.

\section{$R$-(Bi)sectorial operators and the $H^\infty$-functional 
calculus}\label{sec:sectorial}

In this section we present a brief overview of the various notions from 
operator theory used in this paper.

\subsection{$R$-boundedness}

Let $X$ and $Y$ be Banach spaces and let $(r_j)_{j\ge 1}$ be a
sequence of independent {\em Rademacher variables} defined on a probability
space $(\Omega,\mathbb{P})$, i.e., $\mathbb{P}(r_j=
1) = \mathbb{P}(r_j=-1) = \tfrac12$ for each $j$.

A collection of bounded linear operators $\mathscr{T} \subseteq
\Ll(X,Y)$ is said to be {\em $R$-bounded} if there exists a $C\ge 0$ such that
for all $M=1,2,\dots$ and  all choices of $x_1, \ldots, x_M \in X$ and $T_1,
\ldots, T_M  \in \mathscr{T}$ 
 we have $$\E \Big\|\sum_{m=1}^M r_m T_m x_m\Big\|^2
   \leq C^2 \E \Big\|\sum_{m=1}^M r_m  x_m\Big\|^2,$$
where $\E$ denotes the expectation with respect to $\mathbb{P}$.
By considering the case $M=1$ one sees that every $R$-bounded family of 
operators is uniformly bounded. In Hilbert spaces the converse holds,
as is easy to see by expanding the square of the norm as an inner product
and using that $\E r_mr_n = \delta_{mn}$.

Motivated by certain square function estimates in harmonic analysis, 
the theory of $R$-boundedness was initiated in \cite{CPSW} and has found 
widespread use in various areas of analysis, among them parabolic PDE,
harmonic analysis and stochastic analysis. We refer the reader to \cite{DHP, 
HNVW1, HNVW2, KW}
for detailed accounts.
   
\subsection{Sectorial operators}\label{subsec:sectorial}

For $\sigma \in (0,\pi)$ we consider the open sector
$$\Sigma_\sigma^+ := \{ z
\in \CC: \ z \neq 0,\ |\arg z| < \sigma \}.$$ A closed densely defined operator 
$(A,\D(A))$
acting in a complex Banach space $X$ is said to be {\em sectorial}  of angle
$\sigma\in (0,{\pi})$ if $\sigma(A) \subseteq \overline{\Sigma_\sigma^+}$ and
the set $\{ \lambda (\lambda-A)^{-1} : \lambda \notin 
\overline{\Sigma_\vartheta^+}\}$ is bounded
for all $\vartheta \in (\sigma,{\pi})$. The least angle of sectoriality is
denoted by $\omega^+(A)$. If $A$ is sectorial of angle $\sigma\in (0,\pi)$ and the set $\{ \lambda
(\lambda-A)^{-1} : \lambda \notin \overline{\Sigma_\vartheta^+}\}$ is 
$R$-bounded for all
$\vartheta \in (\sigma,{\pi})$, then $A$ is said to be $R$-{\em sectorial} of
angle $\sigma$. The least angle of $R$-sectoriality is
denoted by $\omega_R^+(A).$

\begin{remark}\label{rem} We wish to point out that most authors (including 
\cite{DHP, HNVW2, KW}) impose the 
additional requirements that $A$ be injective and have dense range. In the 
setting 
considered here this would be inconvenient: already in the special case of the 
Ornstein-Uhlenbeck operator, the kernel is non-empty. It is worth 
noting, however, (see \cite[Proposition 2.1.1(h)]{Haase}) 
that a sectorial operator $A$ on a reflexive Banach space $X$ induces a
direct sum decomposition
 \begin{equation}\label{eq:reflexive}
X = \mathsf{N}(A) \oplus \overline{\mathsf{R}(A)}.
 \end{equation}
The part of $A$ in $\overline{\mathsf{R}(A)}$ is sectorial and injective and has
dense range. Thus, $A$ decomposes into a trivial part
and a part that is sectorial in the more restrictive sense of  \cite{DHP, 
HNVW2, KW}.
Since we will be working with $L^p$-spaces in the reflexive range $1<p<\infty$
the results of \cite{DHP, HNVW2, KW} can be applied along this decomposition.
\end{remark}

The typical example of a sectorial operator is the realisation of the Laplace 
operator 
$\Delta$ in $L^p(\RR^n)$, $1\le p<\infty$, and this operator is $R$-sectorial 
if $1<p<\infty$.
More general examples, including the Laplace-Beltrami operator, are discussed 
in \cite{DHP, HNVW2, KW}.

\subsection{Bisectorial operators}\label{subsec:bisectorial}

The theory of sectorial operators has a bisectorial counterpart. We refer the reader
to \cite{ADM, AMN, Duelli} for more information.
For $0<\sigma<\frac12\pi$ we set $\Sigma_\sigma^-:= -\Sigma_\sigma^+$ and 
$$\Sigma_\sigma ^\pm := \Sigma_\sigma^+ \cup \Sigma_\sigma^-.$$
The set $\Sigma_\sigma ^\pm$ is called the {\em bisector} of angle $\sigma$.
A closed densely defined linear operator $(A,\D(A))$ acting in a complex Banach 
space $X$ is called  
{\em bisectorial} of angle $\sigma$  if $\sigma(A) \subseteq 
\overline{\Sigma_\sigma^\pm}$ and
the set $\{ \lambda (\lambda-A)^{-1} : \lambda \notin 
\overline{\Sigma_\vartheta^\pm}\}$ is bounded
for all $\vartheta \in (\sigma,\frac12{\pi})$. The least angle of 
bisectoriality is
denoted by $\omega^\pm(A)$. If $A$ is bisectorial and the set $\{ \lambda
(\lambda-A)^{-1} : \lambda \notin \overline{\Sigma_\vartheta^\pm}\}$ is $R$-bounded 
for all
$\vartheta \in (\sigma,\frac12{\pi})$, then $A$ is said to be $R$-{\em 
bisectorial} of
angle $\sigma\in (0,\frac12{\pi})$. The least angle of $R$-bisectoriality is
denoted by $\omega_R^\pm(A).$

\begin{remark}
If $A$ is bisectorial (of angle $\vartheta$), then $iA$ is sectorial (of angle 
$\frac12\pi +\vartheta$), and therefore Remark \ref{rem} 
applies to bisectorial operators as well. 
\end{remark}

Typical examples of bisectorial operators are $\pm i \ud/\ud x$ in $L^p(\RR)$ 
and the Hodge--Dirac operator 
$\Bigl(\begin{matrix}
       0 & \nabla\s \\ \nabla & 0 
       \end{matrix}\Bigr)$
on $L^p(\RR^n)\oplus L^p(\RR^n;\CC^n)$, $1\le p<\infty$. These operators are 
$R$-bisectorial if $1<p<\infty$. 

\subsection{The $H^\infty$-functional calculus}

In a Hilbert space setting, the $H^\infty$-functional calculus was introduced in 
\cite{McI}. It was extended
to the more general setting of Banach spaces in \cite{CDMY}. For detailed 
treatments we refer the 
reader to \cite{DHP, Haase, HNVW2, KW}.

Let $H^\infty(\Sigma_\sigma^+)$ be the space of all bounded holomorphic
functions on $\Sigma_\sigma^+$, and let $H^1(\Sigma_\sigma^+)$ denote the
space of all holomorphic
functions $\psi:\Sigma_\sigma^+\to \CC $  satisfying
$$\sup_{|\nu|<\sigma} \int_0^\infty |\psi(e^{i\nu}t)|\, \frac{\ud t}{t} < 
\infty.$$ 
If $A$ is a sectorial operator
and $\psi$ is a function in $H^1(\Sigma_\sigma^+)$ with $0 <
\omega^+(A) < \sigma<\pi$, we may define the bounded operator
$\psi(A)$ on $X$ by the Dunford integral
 $$\psi(A)x := \frac{1}{2\pi i} \int_{\partial \Sigma_{\nu}^+} \psi(z) 
(z-A)^{-1}x \ud z, \quad x \in X, $$
where $\omega^+(A) <\nu< \sigma$ and $\partial{\Sigma_{\nu}^+}$ is parametrised
counter-clockwise. By Cauchy's theorem this definition does
not depend on the choice of $\nu.$

A sectorial operator $A$ on $X$ is said to admit a {\em bounded
$H^\infty(\Sigma_{\sigma}^+)$-functional calculus}, or a {\em bounded
$H^\infty$-calculus of angle $\sigma$},  if there exists a
constant $C_\sigma\ge 0$ such that for all $\psi \in H^1(\Sigma_\sigma^+)\cap
H^\infty(\Sigma_\sigma^+)$ and all $x\ \in X$ we have
 \begin{align*}
  \| \psi(A) x\| \leq C_\sigma \|\psi\|_{\infty} \|x\|,
 \end{align*}
where $\|\psi\|_{\infty}=\sup_{z\in \Sigma_\sigma^+}|\psi(z)|$. The infimum of 
all angles $\sigma$ for which such a constant $C$ exists 
is denoted by $\omega_{H^\infty}^+(A).$ We
say that a sectorial operator $A$ admits a {\em bounded
$H^\infty$-calculus} if it admits a bounded
$H^\infty(\Sigma_{\sigma}^+)$-calculus for some $0<\sigma<\pi$.

Typical examples of operators having a bounded $H^\infty$-calculus 
include the sectorial operators mentioned in Subsection \ref{subsec:sectorial}. In fact 
it requires quite some effort to construct sectorial operators
without a bounded $H^\infty$-calculus, and to this date only rather
artificial constructions of such examples are known.

Replacing the role of sectors by bisectors, the above definitions 
can be repeated for bisectorial operators.
The examples of bisectorial operators mentioned in Subsection 
\ref{subsec:bisectorial} have a bounded $H^\infty$-calculus.

\subsection{$R$-(bi)sectorial operators and bounded $H^\infty$-functional calculi}

The following result is a straightforward generalisation of \cite[Proposition 
8.1]{AMN} and 
\cite[Section H]{ADM} (see \cite[Chapter 10]{HNVW2} for the present 
formulation):

\begin{proposition}\label{prop:ADM} Suppose that $A$ is an $R$-bisectorial 
operator on a Banach space of finite cotype. 
Then
$A^2$ is $R$-sectorial, and for each $\omega \in (0,\frac12\pi)$ the following 
assertions are equivalent: 
\begin{enumerate}
\item[\rm(1)] $A$ admits a bounded $H^\infty(\Sigma_\omega^\pm)$-calculus;
\item[\rm(2)] $A^2$ admits a bounded $H^\infty(\Sigma_{2\omega}^+)$-calculus. 
\end{enumerate}
\end{proposition}

\section{The Witten Laplacian}\label{sec:WittenLaplacian}

Let us begin by introducing some standard notations from differential geometry.
For unexplained terminology we refer to \cite{Grig, Lee}.

Throughout this paper we work on a complete Riemannian manifold $(M,g)$ of 
dimension $n$. 
The exterior algebra over the 
tangent bundle $TM$ is denoted by $$\Lambda TM := \bigoplus_{k=0}^n 
\Lambda^kTM.$$
Smooth sections of $\Lambda^kTM$ are referred to as {\em $k$-forms}. 
We set $$
C_{\rm c}^\infty(\Lambda TM) := \bigoplus_{k=0}^nC_{\rm 
c}^\infty(\Lambda^kTM),$$
where $C_{\rm c}^\infty(\Lambda^kTM)$ denotes the vector space of smooth, 
compactly supported $k$-forms.
The inner product of two $k$-forms $\ud x^{i_1} \wedge 
\cdots \wedge \dd x^{i_k}$ and $\ud x^{j_1} \wedge \cdots \wedge \dd x^{j_k}$ 
is defined, in a coordinate chart $(U,x)$, as
$$
(\dd x^{i_1} \wedge \cdots \wedge \dd x^{i_k})\cdot(\dd x^{j_1} \wedge \cdots 
\wedge \dd x^{j_k}) := \det(g^{i_r j_s})_{r,s},
$$
where $(g^{ij})$ is the inverse of the matrix $(g_{ij})$ representing $g$ in 
the chart $(U,x)$.
This definition extends to general $k$-forms by linearity.
For smooth sections $\omega,\eta$ of  $\Lambda TM$, say $\omega = \sum_{k=0}^n 
\omega^k$ 
and $\eta = \sum_{k=0}^n \eta^k$,
we define  $$\omega\cdot\eta := 
\sum_{k=0}^n \omega^k\cdot\eta^k,$$ and we write 
$|\omega| := (\omega\cdot\omega)^{1/2}.$

We now fix a strictly positive 
function $\rho \in C^\infty(M)$ and 
consider the measure $$m(\dd x) := \rho(x) \ud x$$ on $M$, where $\ud x$ is the 
volume measure.
For $1 \leq p < \infty$, we define $L^p(\Lambda^kTM,m)$ to be the Banach space 
of all measurable $k$-forms for which the norm $$\Vert \omega\Vert _p := 
\Big(\int_M 
|\omega|^p \ud m\Big)^{1/p}$$ 
is finite, identifying two such
forms when they agree $m$-almost everywhere on $M$. Equivalently, we could 
define this space 
as the completion of $C_{\rm c}^\infty(\Lambda^kTM)$ with respect to 
the norm $\Vert \cdot\Vert _p$.  
 Finally, we define 
$$
L^p(\Lambda TM,m) := \bigoplus_{k=0}^n L^p(\Lambda^kTM,m)
$$
and endow this space with the norm $\Vert\cdot\Vert_p$ defined by 
$\Vert \omega\Vert_p = \sum_{k=0}^n 
\Vert \omega^k\Vert_p^p$, where $\omega = \sum_{k=0}^n \omega^k$ for $k$-forms 
$\omega^k$. 
In the case of $p = 2$, we will denote the $L^2(\Lambda^kTM,m)$ inner product 
of two $k$-forms $\omega,\eta \in L^2(\Lambda^kTM,m)$ by
$$
\inp{\omega}{\eta}_\rho := \int_M \omega\cdot\eta~\dd m.
$$
Here, the subscript $\rho$ indicates the dependence of the inner product on the 
function $\rho$. When considering the $L^2(\Lambda^kTM,\dd x)$ inner product, 
we will simply write $\inp{\cdot}{\cdot}$. 

The exterior derivative, defined a priori only on 
$C_{\rm c}^\infty(\Lambda TM)$, is denoted by $\dd $. Its 
restriction as a linear operator from $C_{\rm c}^\infty(\Lambda^kTM)$ to 
$C_{\rm c}^\infty(\Lambda^{k+1}TM)$
is denoted by $\dd_k$. As a 
densely defined operator from $L^2(\Lambda^kTM,m)$ to $L^2(\Lambda^{k+1}TM,m)$, 
$\dd_k$ is easily checked to be closable. With slight abuse of notation, its 
closure will again 
be denoted by $\dd_k$. Its adjoint is well defined as a closed densely defined 
operator from $L^2(\Lambda^{k+1}TM,m)$ to $L^2(\Lambda^{k}TM,m)$. 
We will denote this adjoint operator by $\delta_k$. It maps $C_{\rm 
c}^\infty(\Lambda^{k+1}TM)$ into $C_{\rm c}^\infty(\Lambda^kTM)$.

\begin{remark}\label{rho}
It would perhaps be more accurate to follow the notation used in the Introduction 
and denote the operators $\dd$, $\dd_k$ and 
$\delta_k$ by $\dd_\rho$, $\dd_{\rho,k}$ and $\dd\s_{\rho,k}$ respectively, to 
bring out their dependence on $\rho$, but this would unnecessarily burden the 
notation. 
\end{remark}

In Lemma \ref{generaldiv} below we will state an identity relating
$\delta_k$ to the operator $\dd_k^*$, the adjoint of $\dd_k$ with respect to 
the volume measure $\dd x$. 
For this purpose we need the following definition.
Let $k \in \{1,\ldots,n\}$. Let $\omega$ be a $k$-form and $X$ a smooth vector 
field. 
We define $\iota(X)\omega$ as the $(k-1)$-form given by
$$
\iota(X)\omega(Y_1,\ldots,Y_{k-1}) = \omega(X,Y_1,\ldots,Y_{k-1})
$$
for smooth vector fields $Y_1,\ldots,Y_{k-1}$. We refer to $\iota$ as the {\em 
contraction on 
the first entry with respect to $X$}. The next two lemmas are implicit in 
\cite{Bakry}; we include
proofs for the reader's convenience.

\begin{lemma}\label{contraction}
For all smooth $k$-forms $\omega$ and $(k-1)$-forms $\epsilon$ 
and compactly supported smooth functions $f$ on $M$ we have 
$$
\omega\cdot(\dd f\wedge\epsilon) = \iota(\dd f^*)\omega\cdot\epsilon,
$$
where $\ud f^*$ is the smooth vector field
associated to the 1-form $\ud f$ by duality with respect to the Riemannian metric $g$.
\end{lemma}
\begin{proof}
Working in a coordinate chart $(U,x)$, by linearity it suffices to prove the 
claim for 
$\omega = g\dd x^{i_1} \wedge \cdots \wedge \dd x^{i_k}$ where $1 \leq i_1 < 
\cdots < i_k \leq n$ and $\epsilon = h\dd x^{j_1} \wedge \cdots \wedge \dd 
x^{j_{k-1}}$ where $1 \leq j_1 < \cdots < j_{k-1} \leq n$.
In that case we find 
\begin{align*}
\ & \omega\cdot(\dd f\wedge\epsilon)\\
&=	gh(\ud x^{i_1} \wedge \cdots \wedge \dd x^{i_k})\cdot(\ud f \wedge \dd 
x^{j_1} \wedge \cdots \wedge \dd x^{j_{k-1}})\\
&=	\sum_{r=1}^k (-1)^{r+1}gh(\ud x^{i_r}\cdot \dd f)(\ud x^{i_1} \wedge 
\cdots \wedge \widehat{\dd x^{i_r}}\wedge \cdots \wedge \dd x^{i_k})\cdot(\ud 
x^{j_1} \wedge \cdots \wedge \dd x^{j_{k-1}})\\
&= \iota(\dd f^*)\omega\cdot\epsilon.
\end{align*}
Here the third line follows by recalling that the inner product can be seen as 
the determinant of a matrix, and that we can develop this determinant to the 
row 
of $\dd f$. The last equality follows by simply expanding 
$\iota(\dd f^*)\omega$.
\end{proof}

\begin{lemma}\label{generaldiv}
If $\omega$ is a $k$-form, then
$$
\delta_{k-1}\omega = \dd_{k-1}^*\omega - \iota((\dd\log\rho)^*)\omega
$$
where $\ud(\log\rho)^*$ is the smooth vector field associated to the 1-form 
$\ud(\log\rho)$
by duality with respect to the Riemannian metric $g$.
\end{lemma}
\begin{proof}
Suppose that $\omega$ is a $k$-form. For any $(k-1)$-form $\epsilon$ we have
$$
\begin{aligned}
\inp{\epsilon}{\ud_{k-1}^* \omega - \iota((\dd\log\rho)^*)\omega)}_\rho 
   & = \inp{\rho\epsilon}{\ud_{k-1}^*\omega} - 
\inp{\rho\epsilon}{\iota((\dd\log\rho)^*)\omega}
\\ & = \inp{\dd_{k-1}(\rho\epsilon)}{\omega} - 
\inp{\epsilon}{\iota(\rho(\dd\log\rho)^*)\omega}
\\ & = \inp{\rho\dd_{k-1}\epsilon + \dd\rho\wedge\epsilon}{\omega} - 
\inp{\epsilon}{\iota((\dd\rho)^*)\omega}
\\ & = \inp{\dd_{k-1}\epsilon}{\omega}_\rho
\end{aligned}
$$
where we used that $k$-forms are linear over $C^\infty$ functions to arrive at 
the second line. The last equality follows from 
the previous lemma. The claim now follows.
\end{proof}

\begin{definition}[Witten Laplacian] The {\em Witten Laplacian on $k$-forms}
associated with $\rho$ is the operator $L_k$
 defined on $C_{\rm c}^\infty(\Lambda^kTM)$ as 
$$
L_k := \ud_{k-1}\delta_{k-1} + \delta_k\ud_k.
$$
\end{definition}
In the special case that $\rho \equiv 1$, we recover the Hodge-de Rham 
Laplacian 
 $$\Delta_k = \ud_{k-1}\delta_{k-1} + \delta_k\ud_k.$$
Using Lemma \ref{generaldiv} for 1-forms, we obtain the following identity for 
the Witten Laplacian on functions:
$$
L_0 = \ud_0^*\ud_0 - \iota((\ud\log\rho)^*)\ud_0 = \Delta_0 - 
\ud\log\rho\cdot\ud_0 =  \Delta_0  - \nabla \log\rho\cdot \nabla 
$$
where the second identity follows by duality via the Riemannian inner product.
The {\em Bochner-Lichn\'erowicz-Weitzenb\"ock 
formula}  (cf. \cite[Section 5]{Bakry}) asserts that
\begin{equation}\label{eq:BLW}
\frac12\Delta_0|\omega|^2 = \omega\cdot\Delta_k\omega - |\Nabla\omega|^2 - 
\widetilde Q_k(\omega,\omega),
\end{equation}
where $\widetilde Q_k$ is a quadratic form which depends on the Ricci curvature tensor 
(see \cite[Section 5]{Bakry}). Notice that in \cite{Bakry} there is an 
additional term $\frac{1}{k!}$, which comes from the fact that we define 
$|\Nabla\omega|^2$ in a similar way as for $k$-forms, while \cite{Bakry} defines it in 
the sense of tensors. 

An analogue of \eqref{eq:BLW} may be derived for the Witten Laplacian as follows. 
Firstly, if we expand the above definitions using Lemma \ref{generaldiv}, we 
can express $L_k$ in terms of $\Delta_k$:
\begin{align}\label{eq:Lk-expanded}
L_k = \Delta_k - \dd_k(\iota((\dd\log\rho)^*)\omega) - 
\iota((\dd\log\rho)^*) \ud_k\omega.
\end{align}
Obviously, when $k = 0$ the second term on the right-hand side vanishes, while for $k = n$ the last 
term vanishes.
Inserting \eqref{eq:Lk-expanded} into equation \eqref{eq:BLW} we obtain the following variant 
of the Bochner-Lichn\'erowicz-Weitzenb\"ock formula:
\begin{equation}\label{eq:BLW2}
\frac12L_0|\omega|^2 = \omega\cdot L_k\omega - |\Nabla\omega|^2 - 
Q_k(\omega,\omega) ,
\end{equation}
where 
\begin{equation}\label{eq:Qk}
Q_k(\omega,\omega) = \widetilde Q_k(\omega,\omega) + 
\frac12\dd|\omega|^2\cdot\dd\log\rho - 
\omega\cdot\dd(\iota((\dd\log\rho)^*)\omega) - 
\omega\cdot\iota((\dd\log\rho)^*) \ud\omega.
\end{equation}
As $\widetilde Q_k$ only depends on the Ricci curvature tensor, we see that $Q_k$ 
only depends on the Ricci curvature tensor and the positive function $\rho$. 
One has $Q_0 = 0$, while for $k = 1$ one has $Q_1(\omega,\omega) = 
\Ric(\omega^*,\omega^*) - \Nabla\Nabla\log\rho(\omega^*,\omega^*)$ (see 
\cite{Bakry}). The latter is usually referred to as the {\em Bakry--Emery Ricci curvature}.
In what follows, we will refer to $Q_k$ as the {\em Bakry--Emery Ricci curvature on $k$-forms}.

\subsection{The main hypothesis}

We are now ready to state the key assumption, which is a special case of the 
one in Bakry \cite{Bakry}:

\begin{hypothesis}[Non-negative curvature condition]\label{curvaturecondition}
For all  $k = 1,\ldots,n$ the Bakry--Emery Ricci curvature on $k$-forms is non-negative, i.e., we have $Q_k(\omega,\omega) \geq 0$ 
for all $k$-forms $\omega$.  
\end{hypothesis}

We assume non-negativity of the Bakry--Emery Ricci curvature, rather than its boundedness from below (as done in \cite{Bakry}), as in the case of (negative) lower bounds one obtains inhomogeneous Riesz estimates only (see \cite[Theorem 4.1,5.1]{Bakry}). Also note (see \cite{Bakry})
that to obtain boundedness of the Riesz transform on $k$-forms, not only does one need non-negativity of $Q_k$, but also of $Q_{k-1}$ and $Q_{k+1}$.\\

As an example, we will show what this assumption means in the case of 
$M=\RR^n$. The result of our computation is likely to be known, but 
for the reader's convenience we provide the details of the computation. 
Note that the case $k=1$ is much easier due to the simple coordinate free expression for 
the Bakry--Emery Ricci curvature $Q_1$.
In particular,  we will see that this assumption is satisfied in the case of the 
Ornstein-Uhlenbeck operator on $\RR^n$.

\begin{example}
Let $M = \RR^n$ with its usual Euclidean metric and consider a smooth strictly 
positive function $\rho$ on $\RR^n$. Let $k \in \{1,2\ldots,n\}$. We will 
derive a sufficient on $\rho$ in order that $Q_k(\omega,\omega) \geq 0$ for all 
$k$-forms $\omega$. 

Since $\RR^n$ has zero curvature, $\widetilde 
Q_k(\omega,\omega) = 0$ for all $k$-forms $\omega$. Focussing on the remaining terms
in \eqref{eq:Qk}, we will first 
show that $Q_k$ has the `Pythagorean' property described in \eqref{eq:pyth} below. Suppose 
$$\omega = \omega^{(1)}+\dots+\omega^{(N)},$$ where each $\omega^{(j)}$
is of the form $f^{(j)} \dd x^{i_1^{(j)}} \wedge 
\cdots \wedge \dd x^{i_k^{(j)}}$ with $1\le i_1^{(j)} < \dots< i_k^{(j)}\le n$,
and write $I^{(j)} = \{i_1^{(j)}, \dots, i_k^{(j)}\}$. If the index sets
$I^{(1)}, \dots, I^{(N)}$ are all different, then 
\begin{equation}\label{eq:pyth} Q_k(\omega,\omega) = Q_k(\omega_1,\omega_1)+\dots+ Q_k(\omega_N,\omega_N) .
\end{equation}
To keep notations simple we will prove \eqref{eq:pyth} for the case $N=2$; the reader will
have no difficulty in generalising the argument to general $N$.

So let us take $k$-forms $\omega_1 = f \dd x^{i_1} \wedge \cdots 
\wedge \dd x^{i_k}$, where $1 \leq i_1 < \cdots < i_k \leq n$ and $\omega_2 = 
g\dd x^{j_1} \wedge \cdots \wedge \dd x^{j_k}$, where $1 \leq j_1 < \cdots < j_k 
\leq n$ and suppose that $(i_1,\ldots,i_k) \neq (j_1,\ldots,j_k)$. Now consider 
$\omega = \omega_1 + \omega_2$. Since the set of `elementary' $k$-forms $$\{\dd x^{i_1} \wedge \cdots \wedge \dd 
x^{i_k}:\, 1\leq i_1 < \cdots < i_k \leq n\}$$ is an orthogonal basis for 
$\Lambda^kT\RR^n$ we have $|\omega|^2 = |\omega_1|^2 + |\omega_2|^2$ and 
consequently, $$\dd|\omega|^2\cdot\dd(\log\rho) = 
\dd|\omega_1|^2\cdot\dd(\log\rho) + \dd|\omega_2|^2\cdot\dd(\log\rho).$$ 
Furthermore, for any smooth vector field $X$,
$$
\omega\cdot\ud(\iota(X)\omega) = \omega_1\cdot\dd(\iota(X)\omega_1) + 
\omega_2\cdot\dd(\iota(X)\omega_2) + \omega_1\cdot\dd(\iota(X)\omega_2) + 
\omega_2\cdot\dd(\iota(X)\omega_2)
$$
and 
$$
\omega \cdot \iota(X) \ud\omega = \omega_1 \cdot \iota(X) \ud\omega_1 + \omega_2 
\cdot \iota(X) \ud\omega_2 + \omega_1 \cdot \iota(X) \ud\omega_2 + \omega_2 \cdot 
\iota(X) \ud\omega_1.
$$ 
Now
\begin{align*}
\iota(X) \ud\omega_1 
&=
\sum_{i=1}^n \partial_if\dd x^i(X) \ud x^{i_1} \wedge \cdots \wedge \dd 
x^{i_k}\\
&\qquad +
\sum_{i=1}^n\sum_{l=1}^k (-1)^l\partial_if \dd x^{i_l}(X) \ud x^{i_1} \wedge 
\cdots \wedge \widehat{\dd x^{i_l}} \wedge \cdots \wedge \dd x^{i_k}
\end{align*}
and
\begin{align*}
\dd(\iota(X)\omega_1) 
&=
-\sum_{l=1}^k (-1)^l\partial_if \dd x^{i_l}(X) \ud x^{i_1} \wedge \cdots \wedge 
\widehat{\dd x^{i_l}} \wedge \cdots \wedge \dd x^{i_k}.
\end{align*}
Consequently,
$$
\iota(X) \ud\omega_1 + \dd(\iota(X)\omega_1)  = \sum_{i=1}^n \partial_if\dd 
x^i(X) \ud x^{i_1} \wedge \cdots \wedge \dd x^{i_k}.
$$
By orthogonality we thus obtain that
\begin{align*}
&\omega_2\cdot\dd(\iota((\dd\log\rho)^*)\omega_1) + \omega_2 \cdot 
\iota((\dd\log\rho)^*) \ud\omega_1\\
&=
\omega_2\cdot(\ud(\iota((\dd\log\rho)^*)\omega_1) + 
\iota((\dd\log\rho)^*) \ud\omega_1)\\
&=
\sum_{i=1}^n g\partial_if \dd x^i((\dd \log\rho)^*)(\ud x^{i_1} \wedge \cdots 
\wedge \dd x^{i_k})\cdot(\ud x^{j_1} \wedge \cdots \wedge \dd x^{j_k}) = 0.
\end{align*}
Obviously, the same holds if we interchange $\omega_1$ and $\omega_2$. Putting everything together, we obtain $Q_k(\omega,\omega) = 
Q_k(\omega_1,\omega_1) + Q_k(\omega_2,\omega_2)$.
This concludes the proof of (the case $N=2$ of) \eqref{eq:pyth}.

Now consider a $k$-form $\omega$ of the form $f \dd x^{i_1} \wedge \cdots 
\wedge \dd x^{i_k}$ with $1 \leq i_1 < 
\cdots < i_k \leq n$. To simplify notations a bit we shall suppose that $(i_1,\ldots,i_k) 
= (1,\ldots,k)$.
We compute the three last terms on the right-hand side of \eqref{eq:Qk}. 

As to the first term, from $|\omega|^2 = f^2$ we obtain 
$$\frac12\dd|\omega|^2\cdot\dd(\log\rho) = \sum_{i=1}^n f\partial_i 
f\partial_i(\log\rho).$$

Turning to the second term,
\begin{align*}
\iota((\dd\log\rho)^*) \ud\omega
&=
\sum_{j=1}^n ((\dd\log\rho)^*)^j\iota(\partial_j) \ud\omega\\
&=
\sum_{j=1}^n \sum_{i=k+1}^n  \partial_i f 
\partial_j(\log\rho)\iota(\partial_j) \ud x^i \wedge \dd x^1 \wedge \cdots 
\wedge \dd x^k\\
&=
\sum_{i=k+1}^n  \partial_i f \partial_i(\log\rho) \ud x^1 \wedge \cdots \wedge 
\dd x^k\\
&\qquad +
\sum_{i=k+1}^n \sum_{j=1}^k (-1)^j \partial_i f \partial_j(\log\rho)  \ud x^i 
\wedge \dd x^1 \wedge \cdots \wedge \widehat{\dd x^j} \wedge \cdots \wedge \dd 
x^k.
\end{align*}
Hence
$$
\omega\cdot \iota((\dd\log\rho)^*) \ud\omega = \sum_{i=k+1}^n  f\partial_i f 
\partial_i(\log\rho).
$$

Computing the final term, we have 
\begin{align*}
\iota((\dd\log\rho)^*)\omega
&=
f\sum_{j=1}^n ((\dd\log\rho)^*)_j\iota(\partial_j) \ud x^1 \wedge \cdots \wedge 
\dd x^k\\
&=
f\sum_{j=1}^k (-1)^j \partial_j(\log\rho) \ud x^1 \wedge \cdots \wedge \widehat{\dd x^j} \wedge 
\cdots \wedge \dd x^k.
\end{align*}
From this it follows that
\begin{align*}
\dd(\iota((\dd\log\rho)^*)\omega)
&=
\sum_{j=1}^k \sum_{i=1}^n (-1)^j\partial_i(f\partial_j(\log\rho)) \ud x^i 
\wedge \dd x^1 \wedge \cdots \wedge \widehat{\dd x^j} \wedge \cdots \wedge \dd 
x^k\\
&=
\sum_{j=1}^k \sum_{i=1}^n (-1)^j\partial_if \partial_j(\log\rho) \ud x^i \wedge 
\dd x^1 \wedge \cdots \wedge \widehat{\dd x^j} \wedge \cdots \wedge \dd x^k\\
&\qquad +
\sum_{j=1}^k \sum_{i=1}^n (-1)^jf\partial_i\partial_j(\log\rho) \ud x^i \wedge 
\dd x^1 \wedge \cdots \wedge \widehat{\dd x^j} \wedge \cdots \wedge \dd x^k.
\end{align*}
Noting that only the terms with $i = j$ can contribute a non-zero contribution to the inner 
product with $\omega$, we obtain
$$
\omega \cdot \dd(\iota((\dd\log\rho)^*)\omega) = \sum_{i=1}^k 
f\partial_if\partial_i(\log\rho) + f^2\partial_i^2(\log\rho).
$$
Collecting everything, we find that
\begin{align*}
Q_k(\omega,\omega)
&= 
\widetilde Q_k(\omega,\omega) + \frac12\dd|\omega|^2\cdot\dd\log\rho - 
\omega\cdot\dd(\iota((\dd\log\rho)^*))\omega - 
\omega\cdot\iota((\dd\log\rho)^*) \ud\omega\\
&= 
-f^2\sum_{i=1}^k \partial_i^2(\log \rho).
\end{align*}

We thus see that $Q_k(\omega,\omega) \geq 0$ precisely when $\sum_{i=1}^k 
\partial_i^2(\log\rho) \leq 0$.
Recalling the simplification for notational purposes, we conclude that 
$Q_k(\omega,\omega) \geq 0$ for all $k$-forms $\omega$ precisely if for all $1 
\leq i_1 < \cdots < i_k \leq n$ it holds that
$$
\sum_{r=1}^k \partial_{i_r}^2(\log\rho) \leq 0.
$$

In the special case $\rho(x) = e^{-\frac12|x|^2}$ which corresponds to the 
Ornstein-Uhlenbeck operator this condition is clearly satisfied. Indeed, for 
any $j =1,\ldots,n$ we have $\partial_j^2(\log\rho) = -1$.  
\end{example}

We can use the previous example to consider a more general situation.

\begin{example}
Let $(M,g)$ be a complete Riemannian manifold. Suppose the quadratic form $\tilde Q_k$ depending solely on the Ricci curvature is bounded from below for all $1\le k\le n$, i.e., there exist constants $a_1,\ldots,a_n$ such that for 
all $k$-forms $\omega$ we have $$\widetilde Q_k(\omega,\omega) 
\geq a_k|\omega|^2.$$ Fix $k \in \{1,\ldots,n\}$. In 
normal coordinates around a point $p\in M$, the expression for $Q_k(\omega,\omega)$ at $p$ 
reduces to the one of the previous example.  
Consequently, $Q_k(\omega,\omega) \geq 0$ for any $k$-form $\omega$ if for any $p \in M$ and any $1 \leq i_1 < \cdots < i_k \leq n$ one has 
$\sum_{r=1}^k \partial_{i_r}^2(\log\rho)(p) \leq a_k$, where the last 
expression is in normal coordinates around $p$. 
\end{example}

\subsection{The heat semigroup generated by $-L_k$}

We return to the general setting described at the beginning of this section.
For each $k=0,1,\dots,n$ the operator 
$L_k$ is essentially self-adjoint on 
$L^2(\Lambda^kTM,m)$ (see \cite{Bakry, Strichartz} for the case $\rho\equiv 1$ 
and \cite{Yoshida})
and satisfies $\langle L_k\omega,\omega\rangle_\rho = |\dd_k \omega|^2+ 
|\delta_{k-1}\omega|^2\ge 0$ for all smooth $k$-forms $\omega$. 
Consequently, its closure is a 
self-adjoint operator on $L^2(\Lambda^kTM,m)$. With slight abuse of notation 
we shall denote this closure by $L_k$ again. 
By the spectral theorem,
$-L_k$ generates a strongly continuous contraction
semigroup $$P_t^k := e^{-tL_k}, \quad t\ge 0,$$ on $L^2(\Lambda^kTM,m)$. 

From now on we assume that Hypothesis \ref{curvaturecondition} 
is satisfied. As was shown in \cite{Bakry, Yoshida}, under this assumption 
the restriction of $(P_t^k)_{t\ge 0}$ to 
$L^p(\Lambda^kTM,m)\cap L^2(\Lambda^kTM,m)$ extends to a strongly 
continuous contraction semigroup on $L^p(\Lambda^kTM,m)$ for any $p \in 
[1,\infty)$. These extensions are consistent, i.e., 
the semigroups $(P_t^k)_{t\ge 0}$ on $L^{p_i}(\Lambda^kTM,m)$, $i=1,2$, agree on the  
intersection $L^{p_1}(\Lambda^kTM,m) \cap L^{p_2}(\Lambda^kTM,m)$.

The infinitesimal generator of the semigroup $(P_t^k)_{t\ge 0}$ in $L^p(\Lambda^kTM,m)$ 
will be denoted (with slight abuse of notation) 
by $-L_k$ and its domain by ${\mathsf D}_p(L_k)$. 

As an operator acting in $L^2(\Lambda^kTM,m)$, $L_k$ is the closure of an 
operator defined a priori on
$C_{\rm c}^\infty(\Lambda^kTM)$ and therefore the inclusion 
$C_{\rm c}^\infty(\Lambda^kTM)\subseteq {\mathsf D}_2(L_k)$ trivially holds.
The definition of the domain ${\mathsf D}_p(L_k)$ is indirect, however, and based on the 
fact that 
$L_k$ generates a strongly continuous semigroup on $L^p(\Lambda^kTM,m)$. 
Nevertheless we have:

\begin{lemma}\label{Cc}
$C_{\rm c}^\infty(\Lambda^kTM)$ is contained in ${\mathsf D}_p(L_k)$ for all $1 
< p < \infty$.
\end{lemma}
\begin{proof} We follow the idea of \cite[Lemma 4.8]{MN}. 
Pick an arbitrary  $k$-form $\omega \in C_{\rm c}^\infty(\Lambda^kTM,m)$. Then $\omega \in 
{\mathsf D}_2(L_k)$ (by definition of $L_k$ on $L^2(\Lambda^kTM,m)$) and also 
$\omega\in L^p(\Lambda^kTM,m)$. 
Since $L^p(\Lambda^kTM,m)$ is a reflexive Banach space, a standard result in semigroup theory 
 states that in order to 
show that $\omega \in {\mathsf D}_p(L_k)$ it suffices to show that 
$$\limsup_{t\downarrow 
0} \frac1t\Vert P_t^k\omega - \omega\Vert _p < \infty$$ (see, e.g., \cite{BB}). Note that
$
\frac1t(P_t^k\omega - \omega) = -\frac1t\int_0^t P_s^kL_k\omega\ud s
$
in $L^2(\Lambda^kTM,m)$. However, since $L_k\omega \in C_{\rm 
c}^\infty(\Lambda^kTM)$ (as 
both $\ud$ and $\delta$
map $C_{\rm c}^\infty(\Lambda TM)$ to $C_{\rm 
c}^\infty(\Lambda TM)$), 
we can interpret the integral on the right-hand side as a 
Bochner integral in the Banach space $L^p(\Lambda^kTM,m)$
(see \cite[Chapter 1]{HNVW1}). 
Consequently we may estimate
$$
\begin{aligned}
\frac1t\Vert P_t^k\omega - \omega\Vert _p 	&\leq	
\frac1t\int_0^t\Vert P_s^kL_k\omega\Vert _p\ud s \leq
\frac1t\int_0^t\Vert L_k\omega\Vert _p\ud s = \Vert L_k\omega\Vert _p.
\end{aligned}
$$
But then $\limsup_{t\downarrow 0} \frac1t\Vert P_t^k\omega - \omega\Vert _p 
\leq 
\Vert L_k\omega\Vert _p < \infty$. This proves the claim.
\end{proof}

By the Stein interpolation theorem 
\cite[Theorem 1 on p.67]{Stein}, for $p\in (1,\infty)$ and $k=0,1,\dots,n$
the mapping $t \mapsto P_t^k$ extends analytically to a strongly continuous  
$\Ll(L^p(\Lambda^kTM,m))$-valued mapping $z \mapsto P_z^k$ defined on the 
sector $\Sigma_{\omega_p}$ with $\omega_p = \frac{\pi}{2}(1 - |2/p - 1|)$.
On this sector the operators $P_z^k$ are contractive.
This implies that $L_k$ is sectorial of angle $\omega_p$. 

As explained in \cite[p. 625]{Yoshida} it follows from the general theory of 
Dirichlet forms \cite{Fukushima} that there exists a Markov process
 $(X_t)_{t\geq0}$ such that 
\begin{align}\label{eq:Markov}
P_t^0f(x) = \EE^x(f(X_t))
\end{align}
for all $f \in C_{\rm c}^\infty(M)$. Here $\EE^x$ denotes expectation under the 
law of the process $(X_t)_{t\geq0}$ starting almost surely in $x \in M$. 
Using this together with Hypothesis \ref{curvaturecondition} 
(this corresponds to the assumption made in \cite[eq. (1.2)]{Yoshida}, 
see the explanation preceding the proof of theorem \ref{thm:HinftyLk}), it is 
then shown in \cite[Proposition 2.3]{Yoshida} that there exists a Markov process
$(V_t)_{t\geq0}$ such that $$P_t^k\omega(v) = \EE^v(\omega(V_t))$$ for 
all 
$\omega \in C_{\rm c}^\infty(\Lambda^kTM)$. Here, $\EE^v$ denotes expectation 
under 
the law of the process $(V_t)_{t\geq0}$ starting almost surely in $v\in M$. 

As a consequence of \eqref{eq:Markov} the operators 
$P_t^0$ are positive, in the sense that they send non-negative functions to 
non-negative functions. This, together with the following lemma, allows us to 
show that $L_k$ is in fact $R$-sectorial of angle $< \frac12\pi$.

\begin{lemma}[$R$-sectoriality via pointwise domination]\label{lem:Domination}
Let $M$ be a Riemannian manifold of dimension $n$ equipped with a measure 
$m$. Let $k \in \{0,1\ldots,n\}$ and suppose $A$ and $B$ are sectorial 
operators of angle $<\frac12\pi$ 
on the space $L^p(M,m)$ and $L^p(\Lambda^kTM,m)$ respectively, with $1\le 
p<\infty$. 
Suppose the bounded analytic $C_0$-semigroups $(S_t)_{t\ge 0}$ and $(T_t)_{t\ge 0}$ generated by $-A$ 
and $-B$ satisfy the pointwise
bound $$ |T_t \omega| \le C S_t|\omega|$$ for all $\omega\in 
L^p(\Lambda^kTM,m)$ and $t\ge 0,$
where $C$ is a constant. If the set
$\{(I+sA)^{-1}:\, s>0\}$ is $R$-bounded (in particular, 
if $A$ is $R$-sectorial), then $B$ is $R$-sectorial of angle $<\frac12\pi$.
\end{lemma}

For the proof of this lemma we need the following result.

\begin{lemma}\label{lem:squarefunction}Let $(M,g)$ be a Riemannian manifold of dimension $n$ equipped with a measure 
$m$.
For all $\omega_1,\ldots,\omega_N \in L^p(\Lambda^kTM,m)$ we have
$$
\E \Bigl\Vert \sum_{i=1}^N r_i \omega_i\Bigr\Vert_{L^p(\Lambda^kTM,m)}\eqsim_p  
\Bigl\Vert \Big(\sum_{i=1}^N |\omega_i|^2\Big)^{1/2}\Bigr\Vert_{L^p(M,m)},
$$
where $(r_i)_i$ is a Rademacher sequence; the implicit constant only depends on 
$p$. 
\end{lemma}
\begin{proof}
{\em Step 1} -- First we assume that $\omega_1,\ldots,\omega_N$ are supported 
in a single coordinate chart $(U,x)$. 
With slight abuse of notation we will identify each $\omega_i|_U $ with the 
corresponding $\CC^{d_k}$-valued function on $U$; here, $d_k = \binom{n}{k}$
is the dimension of $\Lambda^k TU$. 

Denote by $G_k^{-1}$ the symmetric, 
positive definite $d_k\times d_k$-matrix with elements 
$$
(G_k^{-1})_{i_1i_2\ldots i_k,j_1j_2\ldots j_k} = (\ud x^{i_1} \wedge \cdots 
\wedge 
\dd x^{i_k})\cdot(\ud x^{j_1} \wedge \cdots \wedge \dd x^{j_k})
$$
where $1 \leq i_1 < \cdots< i_k \leq n$ and $1 \leq j_1 < \cdots < j_k \leq n$.

Since $G_k^{-1}$ is orthogonally diagonalisable, 
we have $G_k^{-1}(p) = Q(p)D(p)Q(p)^T$, where $D(p)$ is diagonal 
with positive 
diagonal entries. Now set $$\eta_i(p): = \sqrt{D(p)}Q(p)^T\omega_i(p)$$ for  
$p \in U$.
By using the Kahane-Khintchine inequality,
\begin{align*}
\E\Bigl\Vert \sum_i r_i\omega_i\Bigr\Vert_{L^p(\Lambda^kTM,m)}^p
& =
\E\Bigl\Vert \sum_i r_i\omega_i\Bigr\Vert_{L^p(\Lambda^kTU,m|_U)}^p\\
& \eqsim_p
\E\Bigl\Vert \sum_i r_i\omega_i\Bigr\Vert_{L^2(\Lambda^kTU, m|_U)}^p\\
&=
\Bigl(\E\int_U \Bigl|\sum_i r_i\omega_i\Bigr|^2\ud m \Bigr)^{p/2}\\
&=
\Bigl(\E\int_{U} \sum_{i,j} r_ir_j (\omega_i\cdot\overline 
\omega_j)_{G_k^{-1}}\ud m\Bigr)^{p/2}\\
&=
\Bigl(\E\int_{U} \sum_{i,j} r_ir_j \omega_i^TG_k^{-1}\overline 
\omega_j\ud m\Bigr)^{p/2}\\
&=
\Bigl(\E\int_{U} \sum_{i,j} r_ir_j \eta_i^T\overline \eta_j\ud 
m\Bigr)^{p/2}\\
&=
\Bigl(\int_{U} \E\Bigl|\sum_i r_i\eta_i\Bigr|^2\ud m\Bigr)^{p/2}\\
&=
\E\Bigl\Vert \sum_i r_i\eta_i\Bigr\Vert_{L^2(U, m|_U;\CC^{d_k})}^p.
\end{align*}
Next, by the square function characterisation of Rademacher 
sums for $\CC^{d_k}$-valued functions,
\begin{align*}
\E\Bigl\Vert \sum_i r_i\eta_i\Bigr\Vert_{L^p(U, m|_U;\CC^{d_k})}^p
& \eqsim_p
\Bigl\Vert \Bigl(\sum_i |\eta_i|^2\Bigr)^{1/2}\Bigr\Vert_{L^p(U,m|_U)}^{p}\\
&=
\Bigl\Vert \Bigl(\sum_i \eta_i^T\overline 
\eta_i\Bigr)^{1/2}\Bigr\Vert_{L^p(U,m|_U)}^p\\
&=
\Bigl\Vert \Bigl(\sum_i \omega_i^TG_k^{-1}\overline 
\omega_i\Bigr)^{1/2}\Bigr\Vert_{L^p(U,m|_U)}^p\\
&=
\Bigl\Vert \Bigl(\sum_i 
\omega_i\cdot\omega_i\Bigr)^{1/2}\Bigr\Vert_{L^p(U,m|_U)}^p\\
&=
\Bigl\Vert \Bigl(\sum_i |\omega_i|^2\Bigr)^{1/2}\Bigr\Vert_{L^p(M,m)}^p.
\end{align*}

\smallskip
{\em Step 2} -- 
We now turn to the general case. Let $(\phi_U)_{U\in\mathscr{U}}$ be a partition 
of unity subordinate to a collection of 
coordinate charts $\mathscr{U}$ covering $M$. Then, using Fubini's theorem and 
the result of Step 1, 
\begin{align*}
\E\Bigl\Vert \sum_i r_i\omega_i\Bigr\Vert_{L^p(\Lambda^kTM,m)}^p
& = 
\E \sum_U \int_M   \phi_U\Bigl|\sum_i r_i\omega_i\Bigr|^p~\dd m\\
& =
\E \sum_U \Bigl\Vert  \sum_i 
r_i\phi_U^{1/p}\omega_i\Bigr\Vert_{L^p(\Lambda^kTM,m)}^p\\
& \eqsim_p
\sum_U \Bigl\Vert \Bigl(\sum_i 
|\phi_U^{1/p}\omega_i|^2\Bigr)^{1/2}\Bigr\Vert_{L^p(M,m)}^p\\
& =
\sum_U \int_M \Bigl(\sum_i |\phi_U^{1/p}\omega_i|^2\Bigr)^{p/2}\dd m\\
& =
\sum_U \int_M \phi_U \Bigl(\sum_i |\omega_i|^2\Bigr)^{p/2}\dd m\\
& =
\int_M  \Bigl(\sum_i |\omega_i|^2\Bigr)^{p/2}\dd m\\
& = 
\Bigl\Vert \Bigl(\sum_i |\omega_i|^2\Bigr)^{1/2}\Bigr\Vert_{L^p(M,m)}^p.
\end{align*}
\end{proof}

\begin{proof}[Proof of Lemma \ref{lem:Domination}]
Upon taking Laplace transforms, the pointwise assumption implies, for 
$\lambda\in \CC$ with $\Re\lambda>0$, 
 $$ |(I+\lambda B)^{-1} \omega| \le C (I+\re\lambda A)^{-1}|\omega|.$$
Hence if $\re\lambda_1,\dots \re\lambda_N >0$, then for all 
$\omega_1,\dots,\omega_N\in L^p(\Lambda^kTM,m)$ 
we find, by Lemma \ref{lem:squarefunction},
\begin{align*}  
\E \Bigl\Vert \sum_{i=1}^N r_i (I+\lambda_i B)^{-1} 
\omega_i\Bigr\Vert_{L^p(\Lambda^kTM,m)}
& \eqsim_p
\Bigl\Vert \Big(\sum_{i=1}^N |(I+\lambda_i B)^{-1} 
\omega_i|^2\Big)^{1/2}\Bigr\Vert_{L^p(M,m)}
\\ & \le C 
\Bigl\Vert \Big(\sum_{i=1}^N [(I+\re\lambda_i A)^{-1} 
|\omega_i|]^2\Big)^{1/2}\Bigr\Vert_{L^p(M,m)}
\\ & \eqsim_p
C \E \Bigl\Vert \sum_{i=1}^N r_i (I+\lambda_i A)^{-1} 
|\omega_i|\Bigr\Vert_{L^p(M,m)}
\\ & \le CR\E \Bigl\Vert \sum_{i=1}^N r_i |\omega_i|\Bigr\Vert_{L^p(M,m)}
\\ & \eqsim_p CR\Bigl\Vert \sum_{i=1}^N |\omega_i|^2\Bigr\Vert_{L^p(M,m)}
\\ & \eqsim_p CR
\E \Bigl\Vert \sum_{i=1}^N r_i \omega_i\Bigr\Vert_{L^p(\Lambda^kTM,m)}.
\end{align*}
Here, $R$ denotes the $R$-bound of the set  $\{(I+sA)^{-1}:\,s>0\}$.
This gives the $R$-boundedness of the set $\{(I+\lambda B)^{-1}:\, \re 
\lambda>0\}$.
A standard Taylor expansion argument allows us to extend this to the 
$R$-boundedness of the set
$\{(I+\lambda B)^{-1}:\, \lambda\in \Sigma_\nu\}$ for some $\nu>\frac12\pi$.
\end{proof}

We now return to the setting considered at the beginning of this section.
Combining the preceding lemmas we arrive at the following result.

\begin{proposition}[$R$-sectoriality of $L_k$]\label{prop:Rsectorial}
Let Hypothesis \ref{curvaturecondition} be 
satisfied. For all $1 < p < \infty$ and $k=0,1,\ldots,n$, the operator $L_k$ is 
$R$-sectorial on $L^p(\Lambda^kTM,m)$ with angle $\omega_R^+(L_k) < \frac12\pi$.
\end{proposition}
\begin{proof}
Fix $1 < p < \infty$. As we have already noted, $-L_k$ generates a strongly 
continuous analytic contraction semigroup on $L^p(\Lambda^kTM)$. 
By \cite{Bakry, Yoshida}, these semigroups satisfy the pointwise bound 
\begin{align}\label{eq:ptw} |P_t^k\omega| \leq P_t^0|\omega|
\end{align}
 for all $\omega \in L^p(\Lambda^kTM,m)$. 
Since the semigroup generated by $-L_0$ is positive, 
$L_0$ is $R$-sectorial by  \cite[Corollary 5.2]{KalWeis}. Lemma 
\ref{lem:Domination} then implies that 
$L_k$ is $R$-sectorial, of angle $<\frac12\pi$. 
\end{proof}

We are now ready to state our first main result.  

\begin{theorem}[Bounded $H^\infty$-calculus for 
$L_k$]\label{thm:HinftyLk}
Let Hypothesis \ref{curvaturecondition} be 
satisfied. For all $1 < p < \infty$ and all $k= 0,1,\ldots,n$, the operator 
$L_k$ has a bounded 
$H^\infty$-calculus on $L^p(\Lambda^kTM,m)$ of angle $<\frac12\pi$.
\end{theorem}

For $k=0$ the proposition is an immediate consequence of \cite[Corollary 
5.2]{KalWeis}; see  
\cite{CD-Functional} for a more detailed quantitative statement. 
For $k=1,\dots,n$ this argument cannot be used and instead we shall apply the 
square function
estimates of \cite{Yoshida}. To make the link between the definitions used in 
that paper and the 
ones used here, we need to make some preliminary remarks. 

In \cite{Yoshida}, the Hodge Laplacian on $k$-forms is defined as
\begin{align}\label{tracehessian}
\widetilde\Delta_k := -\Tr(\Nabla\Nabla).
\end{align}
This is motivated by the fact that on functions this operator agrees with $\Delta_k$ (see 
\cite{Grig}). 
Similarly in \cite{Yoshida} one defines
\begin{align}\label{tracehessian-rho}
\widetilde L_k := \widetilde\Delta_k - \Tr(\Nabla(\log\rho)\otimes\Nabla).
\end{align}
Actually, the definition in \cite{Yoshida} there
differs notationally from \eqref{tracehessian-rho} in that $e^{-\rho}$ is written
for the strictly positive function that we denote by $\rho$. 

Define $$V_k := L_k - \widetilde L_k$$ as a linear operator on $C_{\rm 
c}^\infty(\Lambda^kTM)$ (cf. \cite[eq. (1.2)]{Yoshida}, 
recalling our convention of considering the negative Laplacian). 
We will show in a moment that 
\begin{align}\label{eq:ass-Yosh}\omega\cdot V_k\omega = Q_k(\omega,\omega),
\end{align}
so that Hypothesis \ref{curvaturecondition} can be rephrased as assuming that 
$\omega\cdot V_k\omega\ge 0$. 
This corresponds to the assumption made in \cite[Eq. (1.4)]{Yoshida}. Thus, the results 
from \cite{Yoshida} may be applied in the present situation. 

Turning to the proof of \eqref{eq:ass-Yosh}, first observe that 
$\widetilde\Delta_k$ satisfies 
\begin{equation}\label{eq:YBLW1}
\frac12\widetilde\Delta_0|\omega|^2 = \omega\cdot\widetilde\Delta_k\omega - 
|\Nabla\omega|^2,
\end{equation}
from which it follows that
\begin{equation}\label{eq:YBLW2}
\frac12\widetilde L_0|\omega|^2 = \omega\cdot\widetilde L_k\omega - 
|\Nabla\omega|^2 - \frac12\Tr(\Nabla(\log\rho)\otimes \Nabla|\omega|^2) + 
\omega\cdot\Tr(\Nabla(\log\rho)\otimes\Nabla\omega).
\end{equation}
This can be simplified to 
\begin{equation}\label{eq:YBLW3}
\frac12\widetilde L_0|\omega|^2 = \omega\cdot\widetilde L_k\omega - 
|\Nabla\omega|^2.
\end{equation}
Indeed, in a coordinate chart one has
$$
\begin{aligned}
\frac12\Tr(\Nabla(\log\rho)\otimes\Nabla|\omega|^2)	
&=	\frac12\sum_{j=1}^n \Nabla^j(\log\rho)\Nabla_j|\omega|^2\\
&=	\sum_{j=1}^n \Nabla^j(\log\rho)\Nabla_j\omega \cdot \omega
= \Tr(\Nabla(\log\rho)\otimes\Nabla\omega)\cdot\omega.
\end{aligned}
$$									
Noting that $L_0 = \widetilde L_0$, combining \eqref{eq:BLW2} and 
\eqref{eq:YBLW3} gives
$\omega\cdot V_k\omega = Q_k(\omega,\omega)$
as desired.

\begin{proof}[Proof of Theorem \ref{thm:HinftyLk}]
Fix $1 < p < \infty$. By Proposition \ref{prop:Rsectorial}, $L_k$ is $R$-sectorial on 
$L^p(\Lambda^kTM,m)$ and $\omega_R^+(L_k)<\frac12\pi$.
Pick $\vartheta\in (\omega_R^+(L_k),\frac12\pi)$.
The function $\psi(z) := 
\frac{1}{\sqrt2}\sqrt{z}e^{-\sqrt{z}}$ 
belongs to $H^1(\Sigma_\vartheta^+)\cap H^\infty(\Sigma_\vartheta^+)$. Using 
the substitution $t = s^2$ we see that 
$$
\int_0^\infty |\psi(tL_k)\omega|^2 \,\frac{\dd t}{t} = \int_0^\infty \left| 
\frac{\partial}{\partial t}\Big|_{t = s}e^{-tL_k^{1/2}}\omega \right|^2 s\ud s 
$$
Accordingly, by \cite[Theorem 5.3]{Yoshida},
\begin{equation}\label{eq:sqf-L}
\Vert \omega - E_0^k\omega\Vert _p \lesssim_p \left\Vert\int_0^\infty 
|\psi(tL_k)\omega|^2 \,\frac{\dd t}{t}\right\Vert_p \lesssim_p \Vert 
\omega\Vert _p
\end{equation}
for all $\omega \in C_{\rm c}^\infty(\Lambda^kTM)$, where $E_0^k$ denotes 
projection 
onto the kernel of $L_k$.  
By a routine density argument (using that convergence in the mixed 
$L^p(L^2)$-norm 
implies almost everywhere convergence along a suitable subsequence)
these inequalities extend to arbitrary $k$-forms $\omega\in L^p(\Lambda^kTM,m)$.

Now it is well known that for an $R$-sectorial operator, the square 
function estimate \eqref{eq:sqf-L} implies 
the operator having a bounded $H^\infty$-calculus of angle at most 
equal to its angle of 
$R$-sectoriality (see \cite{KalWei2} or \cite[Chapter 10]{HNVW2}).
\end{proof}

\section{The Hodge--Dirac operator}\label{sec:HodgeDirac}

Throughout this section we shall assume that Hypothesis \ref{curvaturecondition} is in force.
Under this assumption one may check,
using the Bochner-Lichn\'erowicz-Weitzenb\"ock formula \eqref{eq:BLW2} instead 
of \eqref{eq:BLW}, that the results in  
\cite[Section 5]{Bakry} proved for the special case $\rho\equiv 1$ 
carry over to general strictly positive functions $\rho \in C^\infty(M)$. 
Whenever we refer to results from \cite{Bakry} we bear this in mind. 

\begin{definition}[Hodge--Dirac operator associated with $\rho$] 
The {\em Hodge--Dirac operator} associated with $\rho$ is the linear operator
$D$ on $C_{\rm c}^\infty(\Lambda TM)$ defined by 
$$D := \dd + \delta.$$
\end{definition}
As in Remark \ref{rho} it would be more accurate to denote this operator by 
$D_\rho$, but again we prefer to keep the notation simple.

With respect to the decomposition $C_{\rm c}^\infty(\Lambda TM) = \bigoplus_{k=0}^n C_{\rm c}^\infty(\Lambda^kTM)$, $D$  can be represented 
by the 
$(n+1)\times(n+1)$-matrix
\[ D =  \left( \begin{array}{ccccc}
0 & \delta_0 & & & \\
\dd_0 & 0 & \delta_1 & & \\
& \ddots & \ddots & \ddots & \\
& &  \dd_{n-2} & 0 & \delta_{n-1}\\
& & & \dd_{n-1} & 0 \end{array} \right),\] 
From 
$\ud^2 = \delta^2 = 0$ it follows that 
\[ D^2 = \left( \begin{array}{ccc}
L_0 & & \\
 & \ddots & \\
& & L_n\\ \end{array} \right)=:L. \]

\begin{lemma}\label{hodgediracclosable}
For all $1\le p<\infty$ the operator is closable as a densely defined operator 
on $L^p(\Lambda TM,m)$.
\end{lemma}
\begin{proof}
For the reader's convenience we include the easy proof.
Let $(\omega_n)_n$ be a sequence in $C_{\rm c}^\infty(\Lambda TM)$ and suppose 
that 
$\omega_n \to 0$ and $D\omega_n \to \eta$ in $L^p(\Lambda TM, m)$. 
Decomposing along the direct sum we find that $\omega_n^k \to \omega^k$ in 
$L^p(\Lambda^kTM,m)$ for $0\leq k \leq n$ and $\ud_{k-1}\omega_n^{k-1}  + 
\delta_k\omega_n^{k+1} \to \eta^k$ in $L^p(\Lambda^kTM,m)$ for $1 \leq k \leq 
n-1$; for $k=0$ we have $\delta_0\omega_n^1 \to \eta^0$ in 
$L^p(\Lambda^0TM,m)$ and for $k = n$ we have $\ud_{n-1}\omega_n^{n-1} \to 
\eta^n$ in $L^p(\Lambda^nTM,m)$.   

First consider $1 \leq k \leq n-1$, and pick $\phi \in C_{\rm 
c}^\infty(\Lambda^kTM,m)$. 
By H\"older's inequality,
$$
\begin{aligned}
\inp{\eta^k}{\phi}_\rho 	&= \lim_{n\to\infty} 
\inp{\dd_{k-1}\omega_n^{k-1} + 
\delta_k\omega_n^{k+1}}{\phi}_\rho\\
			&= \lim_{n\to\infty} 
\inp{\omega_n^{k-1}}{\delta_{k-1}\phi}_\rho + 
\inp{\omega_n^{k+1}}{\dd_k\phi}_\rho\\
 			&= \inp{0}{\delta_k\phi}_\rho + 
\inp{0}{\dd_k\phi}_\rho\\
			&= 0.
\end{aligned}
$$
This is justified since both $\omega_n^{k+1}$ and $\phi$ are compactly 
supported and therefore belong to $\mathsf{D}_q(\delta_k)$, respectively 
$\mathsf{D}_q(\delta_{k-1})$, with $\frac1p+\frac1q=1$. It follows that $\eta^k 
= 0$ by density. The 
cases $k = 0$ and $k=n$ are treated similarly. We conclude that $\eta^k = 0$ 
for all 
$k$, so $\eta = 0$.
 \end{proof} 
 
With slight abuse of notation we will denote the closure 
again by $D$ and write ${\mathsf D}_p(D)$ for its 
domain in $L^p(\Lambda TM, m)$.
The main result of this section asserts that, under Hypothesis
\ref{curvaturecondition},
for all $1 < p 
< \infty$ the operator $D$ is $R$-bisectorial on $L^p(\Lambda TM, m)$ and has a 
bounded $H^\infty$-calculus
on this space.

Since $L_k$ is sectorial on $L^p(\Lambda^kTM,m)$, $1<p<\infty$, 
its square root is well defined and sectorial. Moreover we have 
$ C_{\rm c}^\infty(\Lambda^kTM) \subseteq {\mathsf D}_p(L_k)\subseteq {\mathsf 
D}_p(L_k^{1/2})$
(cf. Lemma \ref{Cc}). 

\begin{lemma}\label{density}
For all $1 < p < \infty$ and $k=0,1,\ldots,n$, 
$C_{\rm c}^\infty(\Lambda^kTM)$ is dense in ${\mathsf D}_p(L_k^{1/2})$.
\end{lemma}
\begin{proof}
Pick an arbitrary $\omega \in {\mathsf D}_p(L_k^{1/2})$. By  
\cite[Proposition 3.8.2]{ABHN} we have $\omega \in {\mathsf 
D}_p((I-L_k)^{1/2})$. 
From the proof of 
\cite[Corollaries 4.3 and 5.3]{Bakry} we see that there exists a sequence 
$(\omega_n)_n$ in $C_{\rm c}^\infty(\Lambda^kTM)$ such that $(I 
+L_k)^{1/2}\omega_n \to 
(I +L_k)^{1/2}\omega$ in $L^p(\Lambda^kTM,m)$. By  \cite[Lemma's 4.2 and 
5.2]{Bakry} we then find that
$$
\Vert \omega_n - \omega\Vert _{{\mathsf D}_p(L_k^{1/2})} = \Vert \omega_n - 
\omega\Vert _p + 
\Vert L_k^{1/2}(\omega_n - \omega)\Vert _p \lesssim\Vert (I 
+L_k)^{1/2}(\omega_n - 
\omega)\Vert _p
$$
By the choice of the sequence $\omega_n$  the latter tends to $0$ and 
consequently we have
$\omega_n\to \omega$ in ${\mathsf D}_p(L_k^{1/2})$.
\end{proof}

The following result is essentially a restatement of \cite[Theorem 5.1, Corollary 5.3]{Bakry}
in the presence of non-negative curvature. The results in \cite{Bakry} are stated only for
the case $\rho\equiv 1$ and given in the form of 
inequalities for smooth compactly supported $k$-forms. 

\begin{theorem}[Boundedness of the Riesz transform associated with 
$L_k$]\label{Riesztransform}
Let Hypothesis \ref{curvaturecondition} hold. For all $1 < p < \infty$ and 
$k= 0,1,\ldots,n$
we have $${\mathsf D}_p(L_k^{1/2}) = {\mathsf D}_p(\dd_k + \delta_{k-1}),$$ 
and for all $\omega$ in this common domain we have
$$
\Vert L_k^{1/2}\omega\Vert _p\simeq_{p,k}\Vert (\dd_k + \delta_{k-1})\omega 
\Vert _p  .
$$
\end{theorem}
Here, $D_k := \dd_k + \delta_{k-1}$ is the restriction 
of $D$ as a densely defined operator acting from $L^p(\Lambda^kTM,m)$ into 
$L^p(\Lambda TM,m)$.
\begin{proof}
We start by showing that ${\mathsf D}_p(L_k^{1/2}) \subseteq {\mathsf 
D}_p(\dd_k + \delta_{k-1})$ 
together with the estimate 
$$
\Vert (\dd_k + \delta_{k-1})\omega \Vert _p \lesssim_{p,k} \Vert 
L_k^{1/2}\omega\Vert _p.
$$
 
Pick an arbitrary $\omega \in {\mathsf D}_p(L_k^{1/2})$. As 
$C_{\rm c}^\infty(\Lambda^kTM)$ is dense in ${\mathsf D}_p(L_k^{1/2})$ by Lemma 
\ref{density}, we can find a sequence $(\omega_i)_i$ of $k$-forms in this space 
converging 
to $\omega$ in ${\mathsf D}_p(L_k^{1/2})$. By \cite[Theorem 5.1]{Bakry} we then 
find, 
for all $i,j$, 
\begin{align*}
& \Vert \omega_i - \omega_j\Vert _p + \Vert (\dd_k + \delta_{k-1})(\omega_i - 
\omega_j)\Vert _p\\ 	
& \qquad \lesssim \Vert \omega_i - \omega_j\Vert _p + \Vert \dd_k\omega_i - 
\dd_k\omega_j\Vert _p + 
\Vert \delta_{k-1}\omega_i - \delta_{k-1}\omega_j\Vert _p\\
&\qquad \lesssim \Vert \omega_i - \omega_j\Vert _p + \Vert L_k^{1/2}\omega_i - 
L_k^{1/2}\omega_j\Vert _p
\end{align*}
which shows that $(\omega_i)_i$ is Cauchy in ${\mathsf D}_p(\dd_k + 
\delta_{k-1})$. 
By the closedness of $\ud_k + 
\delta_{k-1}$ this sequence converges to some $\eta \in 
{\mathsf D}_p(\dd_k + \delta_{k-1})$. 
Since both 
${\mathsf D}_p(L_k^{1/2})$ and ${\mathsf D}_p(\dd_k + \delta_{k-1})$ are 
continuously embedded 
into $L^p(\Lambda^kTM,m)$ we have $\omega_i 
\to \omega$ and $\omega_i \to \eta$ in $L^p(\Lambda^kTM,m)$, and therefore 
$\eta = \omega$. 
This shows that $\omega \in {\mathsf D}_p(\dd_k + \delta_{k-1})$.
To prove the estimate, by \cite[Theorem 5.1]{Bakry} we obtain, for all $i$,
$$
\Vert (\dd_k + \delta_{k-1})\omega_i\Vert _p \leq \Vert \dd_k\omega_i\Vert _p + 
\Vert \delta_{k-1}\omega_i\Vert _p \leq C_{p,k}\Vert L_k^{1/2}\omega_i\Vert _p.
$$
Since $\omega_i \to \omega$ both in 
${\mathsf D}_p(L_k^{1/2})$ and ${\mathsf D}_p(\dd_k + \delta_{k-1})$, it 
follows that
$$
\Vert (\dd_k + \delta_{k-1})\omega\Vert _p \leq C_{p,k}\Vert 
L_k^{1/2}\omega\Vert _p.
$$

The reverse inclusion and estimate may be proved in a similar manner. Now
one uses that $C_{\rm c}^\infty(\Lambda^kTM)$ is dense in ${\mathsf D}_p(\dd_k 
+ \delta_{k-1})$, 
$\dd_k + \delta_{k-1}$ being the closure of its restriction to $C_{\rm 
c}^\infty(\Lambda^kTM)$.
One furthermore uses the estimate in 
\cite[Corollary 5.3]{Bakry} which holds (with $e = 0$ in the notation of 
\cite{Bakry}) by Hypothesis
\ref{curvaturecondition}. Finally, by definition of the norm on $L^p(\Lambda 
TM,m)$, for all $\omega \in C_{\rm c}^\infty(\Lambda^kTM)$ we have 
\begin{equation}\label{eq:dd} \Vert\dd_k\omega\Vert _p + 
\Vert \delta_{k-1}\omega\Vert _p \eqsim_p \Vert (\dd_k + 
\delta_{k-1})\omega\Vert _p
\end{equation}
noting that $\ud_k \omega\in C_{\rm c}^\infty(\Lambda^{k+1}TM)$ and 
$\delta_{k-1}\omega\in C_{\rm c}^\infty(\Lambda^{k-1}TM)$.
\end{proof}

Our proof of the $R$-bisectoriality of $D$ will be based on $R$-gradient bounds 
to which we turn next. 
We begin with a lemma.

\begin{lemma}\label{lem:intersec}
For all $1 < p < \infty$ and $k=0,1,\ldots,n$ we have 
${\mathsf D}_p(L_k^{1/2})\subseteq {\mathsf D}_p(\dd_k)\cap {\mathsf 
D}_p(\delta_{k-1})$.
\end{lemma}
\begin{proof}
Pick $\omega \in {\mathsf D}_p(L_k^{1/2})$ arbitrarily. As 
$C_{\rm c}^\infty(\Lambda^kTM)$ is dense in ${\mathsf D}_p(L_k^{1/2})$ by Lemma 
\ref{density}, 
we can find a sequence $(\omega_i)_i$ of $k$-forms in this space converging to 
$\omega$ in 
${\mathsf D}_p(L_k^{1/2})$. By \cite[Theorem 5.1]{Bakry} we then find, for all 
$i,j$, 
\begin{align}\label{eq:est-ddk}
\Vert \omega_i - \omega_j\Vert _p + \Vert \dd_k\omega_i - \dd_k\omega_j\Vert _p 
\lesssim 
\Vert \omega_i - \omega_j\Vert _p + \Vert L_k^{1/2}\omega_i - 
L_k^{1/2}\omega_j\Vert _p
\end{align}
which shows that $(\omega_i)_i$ is Cauchy in ${\mathsf D}_p(\dd_k)$. By the 
closedness of 
$\ud_k$ we then find that this sequence converges to some $\eta \in {\mathsf 
D}_p(\dd_k)$. 
As in the proof of Theorem \ref{Riesztransform} we show that $\omega = \eta$. 
It follows that $\omega \in {\mathsf D}_p(\dd_k)$.

This proves the inclusion ${\mathsf D}_p(L_k^{1/2}) \subseteq {\mathsf 
D}_p(\dd_k)$.
The inclusion ${\mathsf D}_p(L_k^{1/2}) \subseteq {\mathsf D}_p(\delta_k)$ is 
proved in the same way.
\end{proof}

Thanks to the lemma, the operators
$$
\dd_kL_k^{-1/2}: \mathsf{R}_p(L_k^{1/2}) \to \mathsf{R}_p(\dd_k), \quad 
L_k^{1/2}\omega 
\mapsto \dd_k\omega
$$
and
$$
\delta_{k-1}L_k^{-1/2}: \mathsf{R}_p(L_k^{1/2}) \to \mathsf{R}_p(\delta_{k-1}), 
\quad 
L_k^{1/2}\omega \mapsto \delta_{k-1}\omega
$$
are well defined, and by Theorem \ref{Riesztransform} combined with the 
equivalence of norms \eqref{eq:dd} 
they are in fact $L^p$-bounded.

It also follows from the lemma that the operators $\ud_k(I + t^2L_k)^{-1}$ and 
$\delta_{k-1}(I + t^2L_k)^{-1}$ are well defined and $L^p$-bounded for all $t 
\in \RR$; indeed,
just note that ${\mathsf D}_p(L_k) \subseteq {\mathsf D}_p(L_k^{1/2})\subseteq 
{\mathsf D}_p(\dd_k) \cap {\mathsf D}_p(\delta_{k-1})$. The next proposition 
asserts that these operators
form an $R$-bounded family:

\begin{proposition}[$R$-gradient bounds]\label{gradientbounds}
Let Hypothesis \ref{curvaturecondition} hold. For all $1 < p < \infty$ and $k =0,1,\ldots,n$  the families of operators
$$
\{t\dd_k(I + t^2L_k)^{-1} :  t > 0\}
$$
and
$$
\{t\delta_{k-1}(I + t^2L_k)^{-1} :  t > 0\}
$$
are both $R$-bounded.
\end{proposition}
\begin{proof}
We will only prove that the first set is $R$-bounded. The $R$-boundedness of 
the other set if proved in 
exactly the same way.

For $t > 0$, standard functional calculus arguments show that 
$$
\begin{aligned}
t\dd_k(I + t^2L_k)^{-1} 	&= 
(\dd_kL_k^{-1/2})((t^2L_k)^{1/2}(I+t^2L_k)^{-1})\\
				&= (\dd_kL_k^{-1/2})(\psi(t^2L_k)),
\end{aligned}
$$
where $\psi(z) = \frac{\sqrt{z}}{1 + z}$. Observe that $\psi \in 
H^1(\Sigma_\vartheta^+)\cap H^\infty(\Sigma_\vartheta^+)$ for any $\vartheta 
\in (0,\frac12\pi)$. 
By a result of \cite{KalWei2} (see also \cite[Chapter 12]{KW}) the set 
$$\{\psi(t^2L_k)\,:t > 0\}$$ is $R$-bounded in 
$\Ll(L^p(\Lambda^kTM,m))$. Since $\ud_kL_k^{-1/2}$ is  bounded, 
it follows that the set $$\{(\dd_kL_k^{-1/2})(\psi(t^2L_k)):\, t> 
0\}$$ is $R$-bounded in $\Ll(L^p(\Lambda^kTM,m),L^p(\Lambda^{k+1}TM,m))$. This 
concludes the proof.
\end{proof}

In order to prove the $R$-bisectoriality of the Hodge--Dirac operator we need 
one 
more lemma, which concerns commutativity rules used in the computation of the 
resolvents of the Hodge--Dirac operator.

\begin{lemma}\label{commutativity}
For all $1 \leq p < \infty$, $k=0,1,\ldots,n$, and $t>0$ the 
following identities hold on ${\mathsf D}_p(\dd_k)$ and ${\mathsf 
D}_p(\delta_k)$ respectively:
$$
(I + t^2L_{k+1})^{-1}\dd_k = \dd_k(I + t^2L_k)^{-1}
$$
and
$$
(I + t^2L_k)^{-1}\delta_k = \delta_k(I + t^2L_{k+1})^{-1}.
$$
Similar identities hold with $(I + t^2L_{k+1})^{-1}$ replaced by $(I + 
t^2L_{k+1})^{-1/2}$ or $P_t^k$.
\end{lemma}
\begin{proof}
We will only prove the first identity; the second is proved in a similar 
manner. The corresponding results for $P_t^k$ can be proved along the same 
lines, or deduced
from the results for the resolvent using Laplace inversion, and in turn the 
identities 
involving $(I + t^2L_{k+1})^{-1/2}$ follow from this.
 
For  
$k$-forms $\omega \in C_{\rm c}^\infty(\Lambda^kTM,m)$ we have 
$P_t^{k+1}\dd_k\omega = 
\dd_kP_t^k\omega$ (see \cite{Bakry}). Here, the right-hand side is well defined 
as $P_t^k\omega \in {\mathsf D}_p(L_k) \subseteq {\mathsf D}_p(\dd_k)$ (which 
holds by analyticity of $P_t^k$). Now pick $\omega \in {\mathsf D}_p(\dd_k)$ 
and let $\omega_n \in C_{\rm c}^\infty(\Lambda^kTM)$ be a sequence converging 
to 
$\omega \in {\mathsf D}_p(\dd_k)$. Such a sequence exists by the definition of 
$\ud_k$ as 
a closed operator. Thus $\omega_n \to \omega$ and $\dd_k\omega_n \to 
\dd_k\omega$ in $L^p(\Lambda^kTM,m)$ respectively $L^p(\Lambda^{k+1}TM,m)$. The 
boundedness of $P_t^k$  and $P_t^{k+1}$ then implies that $P_t^k\omega_n \to 
P_t^k\omega$ and $P_t^{k+1}\dd_k\omega_n \to P_t^{k+1}\dd_k\omega$ in 
$L^p(\Lambda^kTM,m)$ respectively $L^p(\Lambda^{k+1}TM,m)$. As 
$P_t^{k+1}\dd_k\omega_n = \dd_kP_t^k\omega_n$ for every $n$, and as the 
left-hand 
side converges, we obtain that $\ud_kP_t^k\omega_n$ 
converges in $L^p(\Lambda^{k+1}TM,m)$. The closedness of $\dd_k$ shows that 
$P_t^k\omega \in {\mathsf D}_p(\dd_k)$ and that $P_t^{k+1}\dd_k\omega = 
\dd_kP_t^k\omega$.

Taking Laplace transforms on both sides we obtain
$$
(t^{-2} + L_{k+1})^{-1}\dd_k \omega = \dd_k(t^{-2} + L_k)^{-1}\omega
$$
from which one deduces the desired identity.
\end{proof}

\begin{remark}
Although we will not need it, we point out the following consequence of the 
preceding results:
 for all $k=0,1,\dots,n$ we have $${\mathsf D}_p(D_k) = {\mathsf 
D}_p(\dd_k)\cap {\mathsf D}_p(\delta_{k-1})$$
 with equivalent norms. 

To prove this, we note that Lemma \ref{lem:intersec}, combined with the domain equality 
of Theorem \ref{Riesztransform}, gives the inclusion
${\mathsf D}_p(D_k)\subseteq {\mathsf D}_p(\dd_k)\cap {\mathsf 
D}_p(\delta_{k-1})$.
To prove the reverse inclusion we argue as follows. For $\omega\in C_{\rm 
c}^\infty(\Lambda^kTM)$
we observed in \eqref{eq:dd} that  
\begin{equation}\label{eq:dd2}\Vert D_k\omega\Vert _p\eqsim_{p,k} \Vert \dd_k 
\omega\Vert_p + \Vert \delta_{k-1}\omega\Vert _p.
\end{equation}
By Theorem \ref{Riesztransform} and the estimate \eqref{eq:est-ddk} used in the proof of Lemma \ref{lem:intersec}
and its analogue for $\delta_{k-1}$, this equivalence of norms 
extends to arbitrary $\omega\in {\mathsf D}_p(L_k^{1/2})$.

Now let $\omega\in  {\mathsf D}_p(\dd_k)\cap {\mathsf D}_p(\delta_{k-1})$ be 
arbitrary. 
For $t>0$ we have $P_t^k \omega\in  {\mathsf D}_p(L_k) \subseteq {\mathsf 
D}_p(L_k^{1/2})$,
so that 
\begin{equation}\label{eq:dd3}\Vert D_k P_t^k\omega\Vert _p\eqsim_{p,k} \Vert \dd_k P_t^k\omega\Vert_p + 
\Vert \delta_{k-1}P_t^k\omega\Vert _p.
\end{equation}
By Lemma \ref{commutativity} we have $\Vert\dd_k P_t^k\omega \Vert_p = \Vert P_t^{k+1}\dd_k\omega \Vert_p \to \Vert\dd_k \omega \Vert_p$ as $t \downarrow 0$,
and similarly $\Vert\delta_{k-1} P_t^k\omega \Vert_p\to \Vert\delta_{k-1} \omega \Vert _p$.
As a consequence, $P_t^k\omega\to \omega$ in  ${\mathsf D}_p(\dd_k)\cap 
{\mathsf D}_p(\delta_{k-1})$. By \eqref{eq:dd3} and the closedness of $D_k$ we then also have  
$\omega\in \D_p(D_k)$ and 
 $P_t^k\omega \to \omega$ in $\D_p(D_k)$. We conclude that ${\mathsf D}_p(\dd_k)\cap 
{\mathsf D}_p(\delta_{k-1})\subseteq {\mathsf D}_p(D_k)$ and that \eqref{eq:dd2} 
holds for all $\omega\in  {\mathsf D}_p(\dd_k)\cap {\mathsf D}_p(\delta_{k-1})$.
\end{remark}

We now obtain the following result.

\begin{theorem}[$R$-bisectoriality of $D$]\label{RbisectorialityPi}
Let Hypothesis \ref{curvaturecondition} hold. For all $1 < p < \infty$ the 
Hodge--Dirac operator $D$ is $R$-bisectorial on $L^p(\Lambda TM, m)$.
\end{theorem}
\begin{proof}
We will start by showing that the set $\{it:\, t\in \RR,\,t \neq 0\}$ is 
contained in the 
resolvent set of $D$. We will do this by showing that $I - itD$ has a 
two-sided bounded inverse given by
$$
\resizebox{\linewidth}{!}{
$
\displaystyle
\left(\begin{array}{cccc}
(I + t^2L_0)^{-1} & it\delta_0(I +t^2L_1)^{-1} & &  \\
it\dd_0(I + t^2L_0)^{-1} & (I+t^2L_1)^{-1} & it\delta_1(I +t^2L_2)^{-1} & \\
\qquad \qquad \qquad \ddots &  & \ddots \qquad \qquad \qquad  \qquad \ddots&  \\
&  it\dd_{n-2}(I + t^2L_{n-2})^{-1} & (I + t^2L_{n-1})^{-1} & it\delta_{n-1}(I 
+ 
t^2L_n)^{-1} \\
& &  it\dd_{n-1}(I + t^2L_{n-1})^{-1} & (I + t^2L_n)^{-1}\\ \end{array} \right)
$
}
$$
with zeroes in the remaining entries away from the three main diagonals.
By the $R$-sectoriality of $L_k$ (Proposition \ref{prop:Rsectorial}) and the 
$R$-gradient bounds (Proposition \ref{gradientbounds}) all entries are bounded. 
It only remains to check that this matrix defines a two-sided inverse of $I-itD$. Let us first 
multiply with $I - itD$ from the left. It suffices to compute the three 
diagonals, as the other elements of the product clearly vanish. It is easy to 
see 
that the $k$-th diagonal element becomes 
\begin{equation}\label{eq:resolventD}
\begin{aligned}
\ & t^2\dd_{k-2}\delta_{k-2}(I + t^2L_{k-1})^{-1} + (I + t^2L_{k-1})^{-1} + 
t^2\delta_{k-1}\dd_{k-1}(I + t^2L_{k-1})^{-1}  
\\ & \qquad = (I + t^2L_{k-1})(I + t^2L_{k-1})^{-1}  = I
\end{aligned}
\end{equation}
using that $L_{k-1} = -(\ud_{k-2}\delta_{k-2} + \delta_{k-1}\dd_{k-1})$; 
obvious adjustments need to be made for $k = 1$ and $k=n$. 
For the two other 
diagonals it is easy to see that one gets two terms which cancel.

To make this argument rigorous, note that both $\dd_{k-2}\delta_{k-2}(I + 
t^2L_{k-1})^{-1}$
and $\delta_{k-1}\dd_{k-1}(I + t^2L_{k-1})^{-1} $ are well defined as bounded 
operators, so that it suffices to 
check the computations for $\omega\in C_{\rm c}^\infty(\Lambda TM)$. The 
asserted well-definedness and
boundedness of the first of these operators can be seen by noting that 
$$\dd_{k-2}\delta_{k-2}(I + t^2L_{k-1})^{-1} = \dd_{k-2}(I + t^2L_{k-2})^{-1/2} 
\circ \delta_{k-2}(I + t^2L_{k-1})^{-1/2},$$
using Lemma \ref{commutativity}; the boundedness of the other operator follows 
similarly.

If we multiply with $I - itD$ from the right and use Lemma 
\ref{commutativity}, we easily see that the product is again the identity.

It remains to show that the set $\{it(it - D)^{-1}: t \neq 0\} = \{(it - 
D)^{-1}: t\neq 0\}$ is $R$-bounded. For this, observe that the diagonal 
entries are $R$-bounded by the $R$-sectoriality of $L_k$. The $R$-boundedness 
of the other entries follows from the $R$-gradient bounds (Proposition 
\ref{gradientbounds}). Since a set of operator matrices is 
$R$-bounded precisely when each entry is $R$-bounded, we conclude that $D$ is 
$R$-bisectorial.
\end{proof}

\begin{proposition}\label{prop:LD2}
Let $1 < p < \infty$. Then $D^2 = L$ as densely defined closed
operators on $L^p(\Lambda TM, m)$.
\end{proposition}
This result may seem obvious by formal computation, but the issue is to 
rigorously
justify the matrix multiplication involving products of unbounded
operators.
\begin{proof}
It suffices to show that $\mathsf{D}_p(L) \subset \mathsf{D}_p(D^2)$ and 
$D^2(I+t^2L)^{-1} = L(I+t^2L)^{-1}$, or equivalently,  
$
(\dd_{k-1}\delta_{k-1} + \delta_k\ud_k)(I+t^2L_k)^{-1} = L_k(I+t^2L_k)^{-1}$ 
for all $k=0,1,\dots,n$. 
The rigorous justification of the equivalent identity \eqref{eq:resolventD}
has already been given in the course of the above proof.

If $\omega \in \mathsf{D}_p(D^2)$, then by Lemma \ref{commutativity} we find 
$$
D^2(I + t^2L)^{-1}\omega = (I + t^2L)^{-1}D^2\omega \to D^2\omega, \qquad t\to 
0.
$$
Here we used that $(I + t^2L)^{-1}$ converges to $I$ strongly as $t\to 0$ by 
the general theory of sectorial operators.  But then we find that
$$
L(I + t^2L)^{-1}\omega = D^2(I + t^2L)^{-1}\omega \to D^2\omega, \qquad t\to 0.
$$ 
As $(I + t^2L)^{-1}\omega \to \omega$ as $t\to 0$, the closedness of $L$ 
gives $\omega \in \D(L)$ and $L\omega = D^2\omega$.
\end{proof}

We are now ready to prove that $D$ has a bounded $H^\infty$-calculus on 
$L^p(\Lambda TM, m)$.

\begin{theorem}[Bounded $H^\infty$-functional calculus for
$D$]\label{bddfunctionalcalculus}
Let Hypothesis \ref{curvaturecondition} hold. For all  $1 < p < \infty$
the Hodge--Dirac operator $D$ on $L^p(\Lambda TM, m)$ has a bounded 
$H^\infty$-calculus on a bisector.
\end{theorem}

\begin{proof}
With all the preparations done, this now follows by combining Proposition 
\ref{prop:ADM} with
Theorems \ref{thm:HinftyLk} and \ref{RbisectorialityPi} and Proposition \ref{prop:LD2}. 
\end{proof}

\smallskip
{\em Acknowledgement} -- The authors thank Alex Amenta and Pierre Portal for 
helpful comments.

\end{document}